\documentclass[11pt]{amsart}
\usepackage{amssymb,amsmath,amsthm, hyperref, verbatim,color}

\newtheorem{theorem}{Theorem}[section]

\newtheorem{lemma}[theorem]{Lemma}

\newcommand{\Q}{\mathbb{Q}}
\newcommand{\Z}{\mathbb{Z}}
\newcommand{\R}{\mathbb{R}}
\newcommand{\C}{\mathbb{C}}
\newcommand{\mf}[1]{\mathfrak{#1}}
\newcommand{\mc}[1]{\mathcal{#1}}
\DeclareMathOperator{\SL}{SL}\DeclareMathOperator{\GL}{GL}\DeclareMathOperator{\diag}{diag}
\DeclareMathOperator{\Sp}{Sp}\DeclareMathOperator{\Tr}{Tr}
\DeclareMathOperator{\Sym}{Sym}\DeclareMathOperator{\Mat}{Mat}

\title{Pullbacks of Siegel Eisenstein Series and Weighted Averages of Critical $L$-values}
\date{\today}
\author{Nadine Amersi}
\address{Department of Mathematics\\ University College London\\ London, UK WC1E 6BT}
\email{n.amersi@ucl.ac.uk}
\author{Jeffrey Beyerl}
\address{Department of Mathematical Sciences\\
Clemson University\\ Clemson, SC 29634}
\email{jbeyerl@clemson.edu}
\author{Jim Brown}
\address{Department of Mathematical Sciences\\
Clemson University\\ Clemson, SC 29634}
\email{jimlb@clemson.edu}
\author{Allison Proffer}
\address{Department of Mathematics and Applied Mathematics\\ Virginia Commonwealth University\\ Richmond, Virginia 23284}
\email{profferar@vcu.edu}
\author{Larry Rolen}
\address{Department of Mathematics\\ University of Wisconsin - Madison\\ Madison, WI 53706}
\email{lrolen@wisc.edu}

\thanks{The first, second, fourth and fifth author were supported by NSF grant DMS-0552799.}
\thanks{The authors would like to thank Paul Garrett for valuable suggestions on an earlier version of this manuscript as well as Neil Calkin and Kevin James for their support during the project.}

\begin{document}
\maketitle

\begin{abstract}
In this paper we obtain a weighted average formula for special values of $L$-functions attached to normalized elliptic modular forms of weight $k$ and full level.  These results are obtained by studying the pullback of a Siegel Eisenstein series and working out an explicit spectral decomposition.
\end{abstract}

\section{Introduction and Statement of Results}

Many of the most influential results and conjectures in number theory can be phrased in terms of $L$-functions, natural generalizations of the Riemann zeta function. For example, they were used by Dirichlet to prove his namesake theorem on primes in arithmetic progressions.  In particular, he studied the non-vanishing properties of certain $L$-functions to deduce his theorem.  Another famous example is the celebrated conjecture of Birch and Swinnerton-Dyer.  This conjecture relates the order of vanishing at $s=1$ of the $L$-function associated to an elliptic curve to the rank of the group of rational points on the elliptic curve.   Given a $L$-function that is ``motivic'', Deligne has conjectured that for certain special values of this $L$-function there should be a normalization of these values that are algebraic.  This conjecture is known in many instances, in particular, it is known for elliptic modular forms where the appropriate periods are known (see \cite{ManinSbornik73} or \cite{ShimuraMathAnn77} for example.) Such special values arise in such deep conjectures as the Bloch-Beilinson conjecture and the Bloch-Kato conjecture.

In this paper we establish weighted averages for certain special values of $L$-functions attached to elliptic modular forms.  This is accomplished by considering pullbacks of Eisenstein series to embedded copies of lower-dimensional Siegel upper-half spaces. In particular, for a partition $n=n_1+n_2+\ldots +n_{m}$ of $n$ we have an embedding $\mathfrak{h}^{n_1}\times\mathfrak{h}^{n_2} \times \cdots \times \mathfrak{h}^{n_m}\hookrightarrow \mathfrak{h}^{n}$ given by $(z_1, \dots, z_{m}) \mapsto \diag(z_1, \dots, z_{m})$.   Thus, given an Eisenstein series on $\mf{h}^{n}$ we can restrict it to the image of $\mathfrak{h}^{n_1}\times\mathfrak{h}^{n_2} \times \cdots \times \mathfrak{h}^{n_m}$ in $\mf{h}^{n}$.  We call this a pullback of the Eisenstein series.  Note such a pullback is a Siegel modular form in $n_i$ dimensions when all other variables are held fixed. The cases $\mathfrak{h}^1\times\mathfrak{h}^1\hookrightarrow \mathfrak{h}^{2}$ and $\mathfrak{h}^1\times\mathfrak{h}^1\times \mathfrak{h}^1 \hookrightarrow \mathfrak{h}^{3}$ were used by Lanphier in \cite{LanphierMathComp10} to calculate averages of special values of symmetric square and symmetric cube $L$-functions.  The method used by Lanphier is closely followed in this paper.

Heim used pullback formulas to discover new congruences for the Ramanujan $\tau$ function related to the irregular primes 131 and 593, see \cite{HeimMathComp09}. In this paper he also gives a new proof of the famous ``691 congruence",
\[
\tau(n)\equiv \sigma_{11}(n) \text{ (mod 691)}
\]
for all $n \geq 1$ where $\tau$ is the Ramanujan tau function and $\sigma_{k}(n):=\displaystyle \sum_{d\vert n}d^k$.

In this paper we study the embedding $\mathfrak{h}^1\times\mathfrak{h}^2\hookrightarrow \mathfrak{h}^{3}$.  Let $H(a,b)$ denote the generalized class numbers of Cohen (\cite{cohen75}).  We define theta functions
\begin{align*}
\vartheta_1(z) &= \sum_{(m,n)\in \mathbb{Z}^2} q^{m^2+n^2},\\
\vartheta_2(z) &= \sum_{(m,n)\in \mathbb{Z}^2}q^{m^2+mn+n^2}.
\end{align*}
Given two modular forms $f$ and $g$, we write $L(s,f,g)$ to denote the Rankin convolution $L$-function of $f$ and $g$. Finally, set $L(s,\Sym^2 f)$ to be the symmetric square $L$-function of $f$.  Our main result is the following theorem.

\begin{theorem}\label{thm:1.1}
Let $\mathcal{B}_k$ be an orthogonal basis of $S_{k}(\SL_2(\Z))$ consisting of normalized cuspidal eigenforms. Then we have the following:
\begin{align*}
\sum_{f \in \mc{B}_{k}} \frac{L(k-1,f)}{\langle f, f \rangle} \mc{A}_{k}(f) & = \left(\frac{2^{4k-4} \pi^{3k-3}}{(k-1)! (2k-3)!}\right)\left(\frac{k}{2B_{k}} \sum_{b=-2}^{2} H(k-1, 4-b^2) + B_{2k-2} \right.\\
        &\left.+(k-1)(2^{2k-4} + 2\cdot 3^{k-1} + 2^{k+2} - 23 - 2^3 H(k-1,3) - 3 H(k-1,4)) \right)
\end{align*}
where
\begin{align*}
\mc{A}_{k}(f) &= \zeta(k-1)\left(3+\frac{(-1)^{k/2}(k-1)!(2\pi)^{k-1}}{(2k-2)!L(2k-2,\Sym^2 f)}\left[2^{2k-3}L(k-1,\chi_{-4})L(k-1,f,\vartheta_1) \right. \right.\\ & \left. \left. + 2 \cdot 3^{k-3/2}L(k-1,\chi_{-3})L(k-1,f,\vartheta_2)\right]\right).
\end{align*}
\end{theorem}

Note that if $k\in\{12,16,18,20,22\}$ then $\dim(S_k(\SL_2(\Z)))=1$ and hence the weighted average becomes an exact value involving $L$-functions associated to the unique normalized cuspidal eigenform in this dimension.

As $k$ is even,  $(-1)^{k-1}\cdot3$ and $(-1)^{k-1}\cdot4$ are fundamental  discriminants. In this situation, the class numbers and the special values of the $L$-functions attached to quadratic field extensions in Theorem \ref{thm:1.1} are given by the following result.

\begin{lemma}(\cite{cohen75})
If $D:=(-1)^r\cdot n$ is a fundamental discriminant, then $H(r,n)=L(1-r,\chi_D)=-\frac{B_{r,\chi_D}}{r},$ where $B_{r,\chi_D}:=\vert D\vert^{r-1}\displaystyle \sum_{j=0}^{\vert D\vert}\chi_D(j)B_r\left(\frac{j}{\vert D\vert}\right)$ is a generalized Bernoulli number and $B_r(x):=\displaystyle\sum_{k=0}^r \binom{r}{k}B_kx^{r-k}$ is a Bernoulli polynomial.
\end{lemma}

The values of $L(k-1,f,\vartheta_i)$ can be computed following the method of \cite{cohen75} and \cite{ZagierSpringer77} along with the trace map as in \cite{bocherer81}. The following tables give the relevant class number values for the first $6$ values of $k$, as well as the value $\alpha_k$ of the weighted average given in Theorem \ref{thm:1.1}.  Note that these values are calculated by evaluating the right hand side of the equation in Theorem \ref{thm:1.1}. Furthermore, observe that the value is zero for $k=14$ as it should be given that there are no nonzero cusp forms of weight $14$.

\begin{table}[ht]
\caption{Generalized Class numbers $H(k-1,n)$} 
\scriptsize
\begin{tabular}{c c c }  
\hline\hline                        
 &$n=3$&$n=4$\\ [0.5ex] 
$k=12$&-3694/3&-50521/2\\
$k=14$&111202/3&2702765/2\\
$k=16$&-13842922/9&-199360981/2\\
$k=18$&252470402/3&19391512145/2\\
$k=20$&-17612343854/3&-2404879675441/2\\
$k=22$&4577258092006/9&370371188237525/2\\

\hline                    

\hline     
\end{tabular}
\label{table:nonlin}  
\end{table}

\begin{table}[ht]
\caption{Values of $\alpha_k$} 
\scriptsize
\begin{tabular}{c c }  
\hline\hline                        
$k$&$\alpha_k$\\ [0.5ex] 
$k=12$&$(2^{31}\cdot \pi^{33})/(3^6\cdot 5^3\cdot 7^3\cdot 11^2\cdot 13\cdot 17\cdot 19\cdot 23\cdot 691)$\\
$k=14$&$0$\\
$k=16$&$(2^{40}\cdot \pi^{45})/(3^{13}\cdot 5^6\cdot 7^3\cdot 11^2\cdot 13^2\cdot 17\cdot 19\cdot 23\cdot 29\cdot 31\cdot 3617)$\\
$k=18$&$(2^{37}\cdot \pi^{51})/(3^{12}\cdot 5^5\cdot 7^5\cdot 11^3\cdot 13^2\cdot 17^2\cdot 19\cdot 23\cdot 29\cdot 31\cdot 43687)$\\
$k=20$&$(2^{39}\cdot \pi^{57})/(3^{17}\cdot 5^7\cdot 7^3\cdot 11^2\cdot 13^2\cdot 17^2\cdot 19^2\cdot 29\cdot 31\cdot 37\cdot 283\cdot 617)$\\
$k=22$&$(2^{42}\cdot 4409\cdot \pi^{63})/(3^{21}\cdot 5^8\cdot 7^5\cdot 11^3\cdot 13^2\cdot 17^2\cdot 19^2\cdot 23\cdot 29\cdot 31\cdot 37\cdot 41\cdot 131\cdot 593)$\\

\hline                    

\hline     
\end{tabular}
\label{table:nonlin}  
\end{table}

\section{Preliminaries and definitions}

We begin by recalling that the Siegel upper-half space is given by
\begin{equation*}
\mathfrak{h}^n =\{ Z \in \Mat_{n}(\C)| \,\, ^tZ=Z, Z=X+iY,Y>0 \},
\end{equation*}
where $^tM$ denotes the transpose of $M$ and $M>0$ signifies that $M$ is positive-definite. Note that for $n=1$ this reduces to the usual complex upper-half space.

Let $J_n:=\begin{pmatrix} 0_n&-1_n\\1_n&0_n\end{pmatrix}$. The symplectic group $\Sp_{2n}$ is defined by
\begin{equation*}
\Sp_{2n} =\{g\in \GL_{2n} | \,\, ^tg J_n g=J_n\}.
\end{equation*}
The group $\Sp_{2n}(\R)$ acts on $\mathfrak{h}^{n}$ via
\begin{equation*}
\begin{pmatrix}
A&B\\
C&D
\end{pmatrix}
\cdot Z = (AZ+B)(CZ+D)^{-1}.
\end{equation*}
Set $\Gamma_n =\Sp_{2n}(\mathbb{Z})$.

Let $P_{n,r}$ be the parabolic subgroup of $\Sp_{2n}$ given by
\begin{equation*}
P_{2n,r}=\left\{ \begin{pmatrix}A & B \\ 0_{(n-r, n+r)} &D \end{pmatrix}\in \Sp_{2n} \right\}.
\end{equation*}
Note that in the case that $r=0$ we obtain the Siegel parabolic subgroup
\begin{equation*}
P_{2n,0}=\left\{\begin{pmatrix} A &B \\ 0_{n} & D \end{pmatrix} \in \Sp_{2n}\right\}.
\end{equation*}

The Klingen-Eisenstein series of degree $n$ attached to a cusp form $f \in S_{k}(\Gamma_{r})$ (with $k> n + r+1$) is defined by
\begin{equation*}
E_{r,k}^n(z;f)=\sum_{\gamma \in P_{2n,r}(\mathbb{Q})\backslash \Sp_{2n}(\mathbb{Q})}f((\gamma z )^{\ast})j_{n}(\gamma,z)^{-k}
\end{equation*}
where $Z^{\ast}$ denotes the upper left $r \times r$ block of $Z$ and for $\gamma= \begin{pmatrix} A &B \\ C & D \end{pmatrix}$ we set $j_{n}(\gamma,Z)=\det(CZ+D)$.  The case that $r=0$ reduces to the case of the classical Siegel Eisenstein series of weight $k$ and level $\Gamma_{n}$, namely,
\begin{equation*}
E_{0,k}^{n}(z; 1) = \sum_{\gamma \in P_{2n,0}(\mathbb{Q}) \backslash \Sp_{2n}(\mathbb{Q})}
j_{n}(\gamma,z)^{-k}.
\end{equation*}
As this Eisenstein series will arise often, we set $E_{n,k}(z) = E_{0,k}^{n}(z;1)$ to ease the notation.

We will also make use of the Siegel operator.  The Siegel operator $\Phi$ is a linear operator sending modular forms of weight $k$ and level $\Gamma_{n}$ to modular forms of weight $k$ and level $\Gamma_{n-1}$.  It is defined by setting
\begin{equation*}
\Phi f(Z) = \lim_{t\rightarrow\infty} f\begin{pmatrix}Z&0\\0&it\end{pmatrix}
\end{equation*}
for $f\in M_k(\Gamma_n)$ and $Z\in \mathfrak{h}^{n-1}$.  The basic properties of the $\Phi$ operator can be found in \cite{Klingen}.  We summarize some of the relevant ones here.
\begin{enumerate}
\item The kernel of $\Phi$ is precisely the cusp forms:  $\ker(\Phi) = S_{k}(\Gamma_{n})$,
\item \label{rmk:2} The Siegel operator takes Klingen Eisenstein series to Klingen Eisenstein series, namely,
\begin{equation*}
\Phi E_{r,k}^{n}(\ast; f) = \left\{ \begin{array}{ll} E_{r,k}^{n-1}(\ast;f) & \text{for $r<n$}; \\ 0 & \text{for $r =n$.} \end{array}\right.
\end{equation*}
\item The Fourier coefficients of $\Phi f$ are easily determined from the expansion of $f$ as follows. If $f(z)=\sum_{T\geq 0,T\in \Lambda_n}a(T)e^{2\pi i \Tr(TZ)}$, then
\begin{equation*}
\Phi f(z') =\sum_{T'\geq 0,T'\in \Lambda_{n-1}}a\begin{pmatrix}T'&0\\0&0\end{pmatrix}e^{2\pi i \Tr(T'z')} \text{ for } z'\in \mathfrak{h}^{n-1}.
\end{equation*}
\end{enumerate}

We have the following decomposition of the space of Siegel modular forms:
\begin{equation*}
M_k(\Gamma_n)=\mathcal{E}_k(\Gamma_n)\bigoplus S_k(\Gamma_n).
\end{equation*}
Here $\mathcal{E}_k(\Gamma_n)$ denotes the space of Klingen-Eisenstein series of weight $k$ and degree $n$. For more on the general definitions and classical theory of Siegel modular forms one can see \cite{GeerSpringer08} or \cite{KohnenPOSTECH07}.

\section{An Explicit Spectral Decomposition}

We now consider the embedding $\iota:\mathfrak{h}^1\times \mathfrak{h}^2\hookrightarrow \mathfrak{h}^3$.
The corresponding group embedding $\iota:\Sp_2 \times \Sp_4 \hookrightarrow \Sp_6$ is given by
\begin{equation*}
\iota \left( \begin{pmatrix} a_1&b_1\\c_1&d_1\end{pmatrix},\begin{pmatrix} a_2&b_2\\c_2&d_2\end{pmatrix} \right)=\begin{pmatrix} a_1&0&b_1&0\\0&a_2&0&b_2\\c_1&0&d_1&0\\0&c_2&0&d_2\end{pmatrix}.
\end{equation*}
Note that we abuse notation and use $\iota$ to denote each of the embeddings.
These two embeddings are compatible in that for all $\gamma=(\gamma_1,\gamma_2)\in \Sp_2(\R) \times \Sp_4(\R)$ and
$Z=(z,w)\in \mathfrak{h}^1\times \mathfrak{h}^2$ we have $\iota(g(Z))=\iota(g)(\iota(Z)).$
In addition, for $Z= (z,w) \in \mf{h}^1 \times \mf{h}^2$ and $\gamma = (\gamma_1, \gamma_2) \in \Sp_2(\Q) \times \Sp_4(\Q)$ we have that
\begin{equation}\label{eqn:jsplit}
j_{3}(\iota(\gamma), \iota(Z)) = j_{1}(\gamma_1, z) j_{2}(\gamma_2, w).
\end{equation}
We write $E_{3,k}(z,w)$ to denote the pullback $E_{3,k}(\iota(z,w))$.  We now study this Eisenstein series.

Our first step is to give a spectral decomposition for the Eisenstein series as given by Garrett in \cite{GarrettProgMath84}.  We begin by recalling the following double coset decomposition.

\begin{lemma}(\cite{GarrettMathAnn87}, \cite{LanphierMathComp10})
The double coset space
$P_{6,0}(\mathbb{Q})\backslash \Sp_6(\mathbb{Q}) / \iota(\Sp_2(\mathbb{Q}),\Sp_4(\mathbb{Q}))$ has two irredundant representatives given by $1_6$ and $\xi: =\begin{pmatrix} 1_3 & 0_3 \\ \varsigma & 1_3\end{pmatrix}$ where $\varsigma = \begin{pmatrix} 0 & 0 & 1 \\ 0 & 0 & 0 \\ 1& 0 & 0 \end{pmatrix}.$  The isotropy subgroups are given by $\iota (P_{2,0}(\mathbb{Q}) \times P_{4,0}(\mathbb{Q}))$ and
\begin{equation*}
H_1(\mathbb{Q})= \left. \left\{ \iota \left( \begin{pmatrix} \alpha & \beta \\ \gamma & \delta \end{pmatrix}^{\natural},\begin{pmatrix} a&x&y&y'\\0&\alpha&y'&\beta\\0&0&a^{-1}&0\\0&\gamma &x'&\delta \end{pmatrix} \right) \right\vert \begin{pmatrix} \alpha&\beta \\ \gamma & \delta \end{pmatrix}\in \SL_2(\mathbb{Q}),a\in \mathbb{Q}^{\times},\ x,x',y,y'\in \mathbb{Q} \right\}
\end{equation*}
where $g^{\natural} = w_2gw_2$ for $w_2 = \begin{pmatrix}0&1\\1&0\end{pmatrix}.$
\end{lemma}

This lemma gives the splitting $P_{6,0}( \mathbb{Q}) \backslash \Sp_6(\mathbb{Q})=\iota (P_{2,0}(\mathbb{Q}),P_{4,0}(\mathbb{Q}))\sqcup H_1(\mathbb{Q})$. Thus, our pullback decomposes as

\begin{align}\label{decomp}
E_{3,k}(z,w)= &\sum_{\gamma\in (P_{2,0}(\mathbb{Q}) \times P_{4,0}(\mathbb{Q})) \backslash ( \Sp_2(\mathbb{Q})\times \Sp_4(\mathbb{Q}))} j_{3}(\gamma, \iota(z,w))^{-k}\\
     &+ \sum_{\gamma \in H_1(\mathbb{Q})\backslash (\Sp_2(\mathbb{Q})\times \Sp_4(\mathbb{Q}))} j_{3}(\xi \gamma,  \iota(z,w))^{-k}.\notag
\end{align}
With regards to the first term, note that  $(P_{2,0}(\mathbb{Q}) \times P_{4,0} (\mathbb{Q}))\backslash (\Sp_2(\mathbb{Q})\times \Sp_4(\mathbb{Q}))\cong (P_{2,0}(\mathbb{Q})\backslash \Sp_2(\mathbb{Q}))\times (P_{4,0}(\mathbb{Q}) \backslash \Sp_4(\mathbb{Q}))$. This combined with equation (\ref{eqn:jsplit}) gives that
\begin{equation*}
\sum_{\gamma\in (P_{2,0}(\mathbb{Q}) \times P_{4,0}(\mathbb{Q})) \backslash (\Sp_2(\mathbb{Q})\times \Sp_4(\mathbb{Q}))} j_{3}(\gamma \iota (z,w))^{-k} =E_{1,k}(z)E_{2,k}(w).
\end{equation*}

The second term is the sum on the ``big cell".  In \cite{brownpullback}, the adelic version of this term is studied in $\S$ 4.  The Eisenstein series in \cite{pullbacks} is assumed to have non-trivial character.  However, this assumption is only used to reduce the Eisenstein series to the sum over the big cell.  Since here we are only looking at the sum over the big cell, the arguments given in $\S$ 4 of \cite{pullbacks} carry-over verbatim to show that
\begin{equation*}
\sum_{\gamma \in H_1(\mathbb{Q})\backslash (\Sp_2(\mathbb{Q})\times \Sp_4(\mathbb{Q}))} j_{3}(\xi \gamma,  \iota(z,w))^{-k}
\end{equation*}
is cuspidal in $z$.

\begin{lemma}(\cite{BochererCrelle85}, Theorem 5) Let $f \in S_{k}(\Gamma_1)$ be a normalized eigenform.  Then we have the following inner product formula:
\begin{equation*}
\left\langle f,E_{3,k}(\ast, w) \right\rangle=\frac{(-1)^{\frac{k}{2}}2^{3-k}\pi \zeta(k-1) L(k-1,f)}{(k-1) \zeta(k)\zeta(2k-2)}E_{1,k}^2(-\bar{w}; f).
\end{equation*}
\end{lemma}

We now wish to ``untwist" the $-\overline{w}$ argument of $E_{1,k}^2(-\bar{w};f)$. For $F$ any Siegel modular form, it is well known that if we write
\begin{equation*}
F^{c}(w) = \sum_{T \in \Lambda, T \geq 0} \overline{a(T)} e^{2 \pi i \Tr(Tw)}
\end{equation*}
where $a(T)$ is the $T^{\text{th}}$ Fourier coefficient of $F$, then $F^{c}(w) = \overline{F(-\bar{w})}$.  The fact that the Hecke operators are self-adjoint with respect to the Petersson product combined with the fact that $f \in S_{k}(\Gamma_1)$ is a normalized eigenform gives that all the eigenvalues of $f$ are real.   From this we see that $L(k-1,f)$ is a real number as well.  Furthermore, we can apply Theorem 2 of \cite{MizumotoInventMath81} to conclude that $E_{1,k}^2(w;f)^{c} = E_{1,k}^2(w;f^{c}) = E_{1,k}^2(w;f)$.  Thus, we have
\begin{align}\label{eqn:innerprod}
\left\langle E_{3,k}(\ast, w), f \right\rangle &= \overline{\left\langle f, E_{3,k}(\ast, w)\right\rangle} \\
        &= \frac{(-1)^{\frac{k}{2}}2^{3-k}\pi \zeta(k-1) L(k-1,f)}{(k-1) \zeta(k)\zeta(2k-2)}\overline{E_{1,k}^2(-\bar{w}; f)} \notag\\
        &= \frac{(-1)^{\frac{k}{2}}2^{3-k}\pi \zeta(k-1) L(k-1,f)}{(k-1) \zeta(k)\zeta(2k-2)}E_{1,k}^2(w; f)^{c}\notag \\
        &= \frac{(-1)^{\frac{k}{2}}2^{3-k}\pi \zeta(k-1) L(k-1,f)}{(k-1) \zeta(k)\zeta(2k-2)}E_{1,k}^2(w; f). \notag
\end{align}

Let $\mc{B}_{k}$ be an orthogonal basis of $S_{k}(\Gamma_1)$ consisting of normalized cuspidal eigenforms. We now choose a convenient basis $\mathcal{B}^2_k$ of $M_k(\Gamma_2)$ as follows. The fact that $\Phi$ is linear combined with property (\ref{rmk:2}) of the $\Phi$ operator given above shows that $\{E_{1,k}^{2}(w;f): f \in \mathcal{B}_{k}\}$ is a linearly independent set of the same dimension as $S_{k}(\Gamma_1)$.   Thus we choose a basis $\mathcal{B}_k^2$ of  $M_k(\Gamma_2)$ by completing $\{ E_{1,k}^{2}(w;f): f \in \mc{B}_{k}\}$ to a full basis. By cuspidality in the small variable and modularity in the big variable, we can expand the final term in equation (\ref{decomp}) as
\begin{equation} \label{E2}
\sum_{f \in \mathcal{B}_k,F \in \mathcal{B}_k^2}c(f,F)f(z)F(w)
\end{equation}
for some $c(f, F) \in \C$. We now use the inner product given in equation (\ref{eqn:innerprod}) to calculate the values of the coefficients $c(f,F)$.

\begin{theorem}\label{thm:spectraldecomp}
\begin{equation}\label{Fourier}
E_{3,k}(z,w)=E_{1,k}(z)E_{2,k}(w)+\frac{(-1)^{\frac{k}{2}}2^{3-k}\pi \zeta(k-1) }{(k-1)\zeta(k)\zeta(2k-2)} \sum_{f\in \mathcal{B}_k}\frac{L(k-1,f)}{\langle f,f\rangle}f(z)E_{1,k}^2(w;f).
\end{equation}
\end{theorem}

\begin{proof} We take the inner product of $E_{3,k}(z,w)$ with a $f_{i} \in \mathcal{B}_{k}$.  Note that since $f_{i}$ is a cuspform, the first term on the right hand side of equation (\ref{decomp}) contributes nothing to the inner product.  This follows immediately using that this term was shown to be $E_{k,1}(z)E_{k,2}(w)$.
Thus, we have
\begin{equation*}
\langle E_{3,k}(z,w),f_{i} \rangle = \left\langle \displaystyle \sum_{\gamma \in H_1(\mathbb{Q}) \backslash (\Sp_2(\mathbb{Q}) \times \Sp_4 (\mathbb{Q}))} j_{3}(\xi \gamma,\iota (z,w))^{-k},f_{i} \right\rangle.
\end{equation*}
Combining this with equation (\ref{E2}) we have
\begin{align*}
\langle E_{3,k}(z,w),f_{i} \rangle &=  \sum_{F \in \mathcal{B}_k^2} \sum_{f_{j} \in \mathcal{B}_k} c(f_j,F) F(w) \langle f_{j},f_{i} \rangle \\
    &= \sum_{F \in \mathcal{B}_{k}^2} c(f_{i},F) F(w) \langle f_{i}, f_{i} \rangle
\end{align*}
where we have used that $\mathcal{B}_{k}$ is an orthogonal basis.  We now use equation (\ref{eqn:innerprod}) to see that we have
\begin{equation}\label{eqn:coeffs}
\frac{(-1)^{\frac{k}{2}}2^{3-k}\pi \zeta(k-1) L(k-1,f)}{(k-1) \zeta(k)\zeta(2k-2)}E_{1,k}^2(w;f_{i})
= \sum_{F \in \mathcal{B}_k^2}  c(f_{i},F) \langle f_i ,f_{i} \rangle F(w).
\end{equation}
Because $E_{1,k}^2(\ast; f_{i})$ is in the basis $\mathcal{B}_k^2$ and the right hand side is a linear combination of the $F \in \mathcal{B}_k^2$, we have that all
of the $c(f_{i},F)$ are zero except when $F=E_{1,k}^2(\ast; f_{i})$. This allows us to diagonalize the expansion of $E_{3,k}(z,w)$ so that we have
\begin{equation*}
E_{3,k}(z,w) = E_{1,k}(z) E_{2,k}(w) + \sum_{f \in \mathcal{B}_{k}} c(f, E_{1,k}^{2}(\ast, f)) f(z) E_{1,k}^{2}(w;f).
\end{equation*}
It only remains to calculate the values of $c(f,E_{1,k}^{2}(\ast; f))$.  However, equating coefficients of $E_{1,k}^2(w;f)$ in equation (\ref{eqn:coeffs}) immediately yields
\begin{equation}\label{coef}
c(f,E_{1,k}^2(\ast; f)) = \frac{(-1)^{\frac{k}{2}}2^{3-k}\pi \zeta(k-1) L(k-1,f)}{(k-1) \zeta(k)\zeta(2k-2) \langle f,f \rangle}.
\end{equation}
\end{proof}

\section{Proof of Theorem \ref{thm:1.1}}

In this section we prove Theorem \ref{thm:1.1}.  This is accomplished by further restricting $E_{3,k}(z,w)$ so that $w$ lies in the image of the embedding $\mf{h}^1 \times \mf{h}^1 \hookrightarrow \mf{h}^2$.  In particular, we are now considering the embedding $\mf{h}^1 \times \mf{h}^1 \times \mf{h}^1 \hookrightarrow \mf{h}^3$.  We abuse notation and denote the elements $\diag(z_1, z_2, z_3)$ as simply $(z_1, z_2, z_3)$.  Applying this restriction to equation (\ref{Fourier}) gives
\begin{equation}\label{eqn:restricteddecomp}
E_{3,k}(z_1,z_2,z_3)=E_{1,k}(z_1)E_{2,k}(z_2,z_3)\\
    +\frac{(-1)^{\frac{k}{2}}2^{3-k}\pi \zeta(k-1) }{(k-1)\zeta(k)\zeta(2k-2)} \sum_{f\in \mathcal{B}_k}\frac{L(k-1,f)}{\langle f,f\rangle}f(z_1)E_{1,k}^2(z_2,z_3; f).
\end{equation}
Set $q_{j} = e^{2 \pi i z_{j}}$.  The proof of Theorem \ref{thm:1.1} now amounts to calculating the Fourier coefficient of the $q_1 q_2 q_3$-term on each side of equation (\ref{eqn:restricteddecomp}) and equating the two.

We begin with the Siegel Eisenstein series $E_{3,k}(z_1, z_2, z_3)$.  The relevant result in this case has already been calculated by Lanphier.

\begin{lemma}(\cite{LanphierMathComp10}, Lemma 6) The Fourier coefficient of the $q_1 q_2 q_3$-term of $E_{3,k}(z_1, z_2, z_3)$ is
\begin{equation*}
\frac{-2^3k}{B_k}+ \frac{2^3k(k-1)}{B_k(B_{2k-2})}(2^3H(k-1,3)+3H(k-1,4))+ \frac{(-1)^{\frac{k}{2}}2^3k(k-1)}{\vert B_kB_{2k-2} \vert}(2^{2k-4}+2 \cdot 3^{k-1}+2^{k+2}-23).
\end{equation*}
\end{lemma}

Now, we turn our attention to the first term in the right hand side of equation (\ref{Fourier}), namely, $E_{1,k}(z_1)E_{2,k}(z_2,z_3)$. It is well known that
\begin{equation*}
E_{1,k}(z_1)=1+ \frac{2}{\zeta (1-k)} \displaystyle \sum_{n=1}^{\infty} \sigma_{k-1}(n)q_1^n.
\end{equation*}
In particular, we see that the $q_1$-coefficient of $E_{1,k}(z_1)$ is
\begin{equation*}
\frac{2}{\zeta(1-k)}.
\end{equation*}

Let $E_{2,k}(w)$ have Fourier expansion given by
\begin{equation*}
E_{2,k}(w) = \sum_{T \in \Lambda , T\geq 0} A_{2,k}(T) e^{2 \pi i \Tr(Tw)},
\end{equation*}
where $\Lambda = \left\{\left(\begin{smallmatrix}a & b/2 \\ b/2 & c \end{smallmatrix} \right) \vert a,b,c \in \Z \right\}$.  Then
the Fourier expansion of $E_{2,k}$ restricted to $\mf{h}^1 \times \mf{h}^1$ is given by
\begin{equation*}
E_{2,k}(z_2, z_3) = \sum_{n_2,n_3 \geq 0} \left( \sum_{T \in \Lambda(n_2,n_3)} A_{2,k}(T) \right) q_2^{n_2} q_3^{n_3}
\end{equation*}
where $\Lambda (n_2,n_3) = \left\{ \left(\begin{smallmatrix} n_2& b/2 \\ b/2 &n_3 \end{smallmatrix} \right) \in \Lambda \vert 4n_2 n_3 - b^2 \geq 0 \right\}$.
As we are interested in the case that $n_2 = n_3 =1$, we see that the only possible values for $b$ are $0, \pm 1, \pm 2$.  Thus, we have that the $q_2 q_3$-coefficient of $E_{2,k}(z_2, z_3)$ is given by
\begin{equation*}
\sum_{b=-2}^{2} A_{2,k} \left(\begin{pmatrix} 1 & b/2 \\ b/2 & 1 \end{pmatrix} \right).
\end{equation*}
The $T^{\text{th}}$ Fourier coefficient of $E_{2,k}$ is given in \cite{eichlerzagier} as
\begin{equation*}
A_{2,k}(T) = \frac{2}{\zeta(3-2k)\zeta(1-k)} \sum_{d \vert (a,b,c)} d^{k-1} H\left(k-1, \frac{4ac-b^2}{d^2}\right)
\end{equation*}
for $T = \begin{pmatrix} a & b/2 \\ b/2 & c \end{pmatrix}.$  Thus, we have that the
$q_2q_3$-coefficient of $E_{2,k}(z_2,z_3)$ is
\begin{equation*}
\frac{2}{\zeta(3-2k)\zeta(1-k)} \sum_{b=-2}^{2} H(k-1, 4-b^2).
\end{equation*}
Thus, we have that the $q_1 q_2 q_3$-coefficient of $E_{1,k}(z_1)E_{2,k}(z_2, z_3)$ is given by
\begin{equation*}
\frac{4}{\zeta(3-2k)\zeta(1-k)^2} \sum_{b=-2}^{2} H(k-1, 4-b^2).
\end{equation*}

It only remains to consider the summation over $\mathcal{B}_{k}$ on the right hand side of equation (\ref{eqn:restricteddecomp}).  As each $f \in \mathcal{B}_{k}$ is a normalized eigenform, we know the $q_1$-coefficient is $1$.  Thus, determining the $q_1 q_2 q_3$-coefficient of these terms comes down to analyzing the $q_2 q_3$-coefficients of the Klingen Eisenstein series $E_{1,k}^2(z_2,z_3;f)$.
We have the Fourier expansion
\begin{equation*}
E_{1,k}^{2}(z_2, z_3;f)= \sum_{n_2,n_3 \geq 0} \left(\sum_{T \in \Lambda(n_2,n_3)} A_{1,k}^2(T;f) \right) q_2^{n_2}q_3^{n_3}.
\end{equation*}
As above, we see that the $q_2 q_3$-coefficient is given by
\begin{equation*}
\sum_{b=-2}^{2} A_{1,k}^2\left(\begin{pmatrix} 1 & b/2 \\ b/2 & 1 \end{pmatrix}; f \right).
\end{equation*}
We have by \cite{MizumotoInventMath81} that the $\begin{pmatrix} 1 & 0 \\ 0 & 0 \end{pmatrix}$-Fourier coefficient of $E_{1,k}^2(\ast; f)$ is equal to the first Fourier coefficient of $f$, i.e., it is 1.  Furthermore, since $\begin{pmatrix} 1&-1 \\ -1&1 \end{pmatrix}$ and $\begin{pmatrix} 1&1 \\ 1&1 \end{pmatrix}$ are both unimodularly equivalent to $\begin{pmatrix} 1&0 \\ 0&0 \end{pmatrix}$,  we conclude that
\begin{equation*}
A_{1,k}^2\left( \begin{pmatrix} 1&-1 \\ -1&1 \end{pmatrix}; f \right) = A_{1,k}^2 \left( \begin{pmatrix} 1&1 \\ 1&1 \end{pmatrix}; f \right) = A_{1,k}^2 \left( \begin{pmatrix} 1&0 \\ 0&0 \end{pmatrix}; f \right) = 1.
\end{equation*}

It only remains to consider the final three possibilities for $T$.
For $T$ so that $-\det(2T)$ is a fundamental discriminant, Theorem 1 in \cite{MizumotoInventMath81} gives that
\begin{equation*}
 A_{1,k}^2(T;f) = (-1)^{\frac{k}{2}} \frac{(k-1)!}{(2k-2)!}(2 \pi)^{k-1} \det(2T)^{k- \frac{3}{2}} \frac{L(k-1, \chi_{-\det(2T)}) L(k-1,f, \vartheta_T) }{L(2k-2,f,\Sym^2)}
 \end{equation*}
 where $\chi_{-\det(2T)}$ is the Dirichlet character associated to the field $\mathbb{Q}(\sqrt{-\det(2T)})$.  Observing that $-\det(2T)$ is a fundamental discriminant for each of the remaining three values of $T$, we use this result to calculate the remaining Fourier coefficients.
 For $T=  \begin{pmatrix} 1&0 \\ 0&1 \end{pmatrix}$, we have
\begin{equation*}
A_{1,k}^2\left(\begin{pmatrix} 1&0 \\ 0&1 \end{pmatrix}; f \right) = (-1)^{\frac{k}{2}} \frac{(k-1)!}{(2k-2)!}(2 \pi)^{k-1} 2^{2k- 3} \frac{L(k-1, \chi_{-4}) L(k-1,f, \vartheta_1) }{L(2k-2,f,\Sym^2)}.
 \end{equation*}
 Finally, the remaining two cases are given by
 \begin{align*}
 A_{1,k}^2 \left(\begin{pmatrix} 1& -\frac{1}{2} \\ -\frac{1}{2}&1 \end{pmatrix}; f \right) &= A_{1,k}^2 \left( \begin{pmatrix} 1&\frac{1}{2} \\ \frac{1}{2}&1 \end{pmatrix}; f \right) \\
 &= (-1)^{\frac{k}{2}} \frac{(k-1)!}{(2k-2)!}(2 \pi)^{k-1}  3^{k- \frac{3}{2}} \frac{L(k-1, \chi_{-3}) L(k-1,f, \vartheta_2)}{L(2k-2,f,\Sym^2)}.
 \end{align*}

Recall the following well-known formulas
\begin{align*}
B_{2n} &= (-1)^{n+1} |B_{2n}| \\
\zeta(2n) &= \frac{(-1)^{n + 1} (2 \pi)^{2n} B_{2n}}{2 (2n)!} \\
\zeta(-n) &= - \frac{B_{n+1}}{n+1}
\end{align*}
for $n \geq 1$.  Using that $k$ is necessarily even, we combine these formulas with the calculations given above and rearrange to obtain Theorem \ref{thm:1.1}.

\bibliographystyle{plain}
 \bibliography{mybib}

\end{document}